\documentclass[12pt,reqno]{amsart}
\newcommand{\gl}{\mathfrak{\mathop{gl}}}

\newcommand{\GL}{\mathrm{\mathop{GL}}}

\newcommand{\C}{\mathbb{C}}
\newcommand{\R}{\mathbb{R}}
\newcommand{\Q}{\mathbb{Q}}

\newcommand{\g}{\mathfrak{g}}
\renewcommand{\a}{\mathfrak{a}}

\newcommand{\ad}{{\mathop{\mathrm{ad}}}}

\newcommand{\h}{\mathfrak{h}}
\newcommand{\y}{\mathfrak{y}}
\newcommand{\cc}{\mathfrak{c}}
\newcommand{\sqq}{\sqrt{-1}}

\usepackage{color}
\usepackage{amssymb}

\newtheorem{theorem}{Theorem}[section]
\newtheorem{lemma}[theorem]{Lemma}

\newtheorem{corollary}[theorem]{Corollary}

\theoremstyle{definition}

\newtheorem{remark}[theorem]{Remark}

\numberwithin{equation}{section}

\begin{document}

\title[Constructive decomposition of real representations]{A constructive
method for decomposing real representations}

\author[S. Ali]{Sajid Ali}

\address{Department of Mathematics and Statistics,
Zayed University, Dubai, United Arab Emirate}

\email{sajid$\_$ali@mail.com}

\author[H. Azad]{Hassan Azad}

\address{Abdus Salam School of Mathematical Sciences, GC University Lahore, 68-B, New Muslim
Town, Lahore 54600, Pakistan}

\email{hassan.azad@sms.edu.pk}

\author[I. Biswas]{Indranil Biswas}

\address{School of Mathematics, Tata Institute of Fundamental
Research, Homi Bhabha Road, Mumbai 400005, India}

\email{indranil@math.tifr.res.in}

\author[W. A. de Graaf]{Willem A. de Graaf}

\address{Dipartimento di Matematica, Via Sommarive 14,
I-38123 Povo (Trento), Italy}

\email{degraaf@science.unitn.it}



\date{}

\begin{abstract}
A constructive method for decomposing finite dimensional representations of
semisimple real Lie algebras is developed. The method is illustrated by an
example. We also discuss an implementation of the algorithm in the language
of the computer algebra system {\sf GAP}4.
\end{abstract}

\maketitle

\tableofcontents

\section{Introduction}\label{se1}

A fundamental result in the theory of Lie algebras and Lie groups says that
a representation of a semisimple Lie algebra over a field of characteristic 0
is completely reducible, that is, it can be written as the direct sum of
irreducible representations. This immediately leads to the algorithmic problem
to compute the irreducible representations appearing in this decomposition.

We consider this problem over the base fields $\C$ and $\R$. Over $\C$ the
algorithm is an immediate consequence of the classification of the irreducible
representations of a semisimple Lie algebra in terms of highest weights.
We briefly review this in Section \ref{se3}.

The theory of finite dimensional real representations of
real semisimple Lie algebras --- due to Cartan, Iwahori and Karpelevich ---
is explained in detail in the books \cite{GG}, \cite{On}. In Section
\ref{se4} we describe our algorithm for decomposing representations of
semisimple Lie algebras over $\R$. It uses the representation theory
of these algebras, and consequently is more complicated than the algorithm in the
complex case. The basic structure of the algorithm is as follows. First
we decompose the representation over $\C$. The resulting irreducible complex
modules are either self-conjugate, or appear in pairs of conjugate modules.
The real points of a self-conjugate module yield an irreducible real module.
Furthermore, the real points in the direct sum of a complex module and its
conjugate also give an irreducible real module. In order to find the
self conjugate complex modules, and the conjugate pairs, we use an
endomorphism related to a special element of the Weyl group.

We have implemented the algorithm in the language of the computer algebra
system {\sf GAP}4 \cite{gap4}, using the functionality of the package
{\sf CoReLG} \cite{corelg}. This implementation will be part of the
new release of that package. 

\section{Preliminaries and notation}\label{se2n}

Throughout this paper $\g$ denotes a real semisimple Lie algebra. Its
complexification $\g^\C \,=\, \g\otimes_\R \C$ is a complex semisimple Lie
algebra. Let $\rho \,:\, \g \,\longrightarrow\, \gl(V)$ be a finite-dimensional real representation
of $\g$. Then its complexification $\rho^\C$ is the representation of
$\g^\C$ in $\gl(V^\C)$, where $V^\C = V\otimes_\R \C$; we recall that $\rho^\C$ is defined by
$$\rho^\C (X+\sqq \cdot Y ) \,=\, \rho(X) + \sqq \cdot \rho(Y) \text{ for } X,Y\,\in\, \g\, .$$
Conversely, starting with a complex representation $\rho^\C \,:\, \g^\C\,\longrightarrow\,
 \gl(V^\C)$ we can view $V^\C$ as a real vector space (of dimension
$2\cdot \dim_\C V^\C$) which we denote by $V^\C_\R$. Then by restricting $\rho^\C$
to $\g$ we obtain a representation $\rho \,:\, \g \,\longrightarrow\, \gl(V^\C_\R)$. This
representation $\rho$ is called the realification of $\rho^\C$.

When convenient, we shall also use the equivalent language of modules. 
If $\rho$ is as above, then $V$ is said to be a $\g$-module, and we also
write $X\cdot v$ in place of $\rho(X)v$ for $X\,\in\, \g$, $v\,\in\, V$. Similarly,
we say that $V^\C$ is a $\g^\C$-module. 

We let $G_\rho$ be the connected Lie group in $\GL(V)$ with Lie algebra
$\rho(\g)$. Similarly, $G_\rho^\C$ is the connected subgroup of $\GL(V^\C)$
with Lie algebra $\rho(\g^\C)$. Then the connected component of the group
of real points of $G_\rho^\C$ containing the identity element is the group
$G_\rho$. The group $G_\rho$ is generated by real 1-parameter subgroups
$$\{\exp (r\rho(X))\,\mid\,
r\,\in\, \mathbb{R}\}\, ,\ \ X\,\in\, \mathfrak{g}\, ,$$
while $G_\rho^{\mathbb{C}}$ is generated by the complex 1-parameter subgroups
$$\{\exp (z\rho^\C(X))\,\mid\,
z\,\in\, \mathbb{C}\}\, ,\ \ X\,\in\, \mathfrak{g}^\C\, .$$

We have $\g^\C \,=\, \g + \sqq \g$, and therefore $\g^\C$ has a conjugation
$$
\sigma\, :\, {\mathfrak{g}}^{\mathbb{C}}\,\longrightarrow\,
{\mathfrak{g}}^{\mathbb{C}}\, ,\ \ X+\sqrt{-1}\cdot Y\, \longmapsto\,
X-\sqrt{-1}\cdot Y\, .
$$
Note that $[\sigma(X),\,\sigma(Y)] \,=\, \sigma([X,\,Y])$ for all $X,\,Y\,\in\, \g^\C$.
Similarly, $V^{\mathbb{C}}$, $\mathfrak{gl}(V^{\mathbb{C}})$ and $G^{\mathbb{C}}$ have
conjugations defined in the obvious way. We shall denote all of them by the same
symbol $\sigma$, and occasionally write $\sigma (v)\,=\,\overline{v}$ etc.

\subsection{Roots and reflections}

Let $\mathfrak{c}$ be a Cartan subalgebra of $\mathfrak{g}$. By definition,
$\mathfrak{c}$ is a maximal abelian subalgebra of $\mathfrak{g}$ which is
diagonalizable in the complexification ${\mathfrak{g}}^{\mathbb{C}}$ of
$\mathfrak{g}$. So $\cc^\C$ is a Cartan subalgebra of $\g^\C$.
A nonzero element $X\,\in\, {\mathfrak{g}}^{\mathbb{C}}$ such that
$$[H,\, X]\,=\,\alpha(H)X$$ for all $H$ in $\mathfrak{c}^\C$ is called
a root vector and the corresponding linear function $$\alpha \,:\, \cc^\C
\,\longrightarrow\, \C$$
a root of the Cartan subalgebra $\mathfrak{c}^\C$. Then the space
$$\g^\C_\alpha \,=\, \{ X\,\in \,\g^\C \,\mid\, [H,\,X] \,=\, \alpha(H) X\ \text{ for all }
H\,\in\, \cc^\C\}$$
is 1-dimensional and it is called the root space of $\alpha$.

A complex number $z\,=\,a+\sqrt{-1}\cdot b$ with $a,\,b\,\in\, \mathbb R$
is said to be positive if either $a\, \, >\, 0$ or $a\, \, =\, 0$
and $b\,>\,0$.

Fixing a basis in $\mathfrak{c}$ defines a cone in the set of complex valued
linear functions on $\mathfrak{c}$. More precisely, if we fix an ordered basis $H_1,\,\cdots
,\, H_r$ of $\mathfrak{c}$, a nonzero complex valued linear function $\lambda$ on
$\mathfrak{c}$ is said to be positive if the first nonzero complex number among
$\lambda (H_1),\, \cdots ,\, \lambda (H_r)$ is a positive complex number. Otherwise
it is called negative.
Positive roots which are not a sum of two positive roots are called simple roots.
Thus a knowledge of all positive roots determines algorithmically the Dynkin
diagram of ${\mathfrak{g}}^{\mathbb{C}}$.

In general, any Borel subalgebra determines a system of positive roots --- without
any reference to a basis of a given Cartan subalgebra of the Borel subalgebra. This
happens typically when one embeds a given subalgebra of ad--nilpotent elements
into a maximal subalgebra of such elements, following the algorithms in
\cite{AABGM}. An example is given in Section 5 of \cite{AABGM}, where a system of type
$G_2$ in ${\mathbb{R}}^2$ with rational coordinates comes up when one embeds a
commuting algebra of translations into a maximal ad--nilpotent subalgebra.

Let $\alpha_1,\,\cdots,\,\alpha_\ell$ be the simple roots. Let $C$ denote the
corresponding Cartan matrix of the root system. Then there are
elements $X_{\pm \alpha_i} \,\in\, \g_{\pm \alpha_i}^\C$, $H_i\,\in\, \cc^\C$ with
\begin{align*}
 & [ H_i,\,H_j ] \,=\, 0\\
 & [ X_{\alpha_i},\, X_{-\alpha_j} ] \,=\, \delta_{ij} H_i\\
 & [ H_i, \,X_{\pm \alpha_j} ] \,= \,\pm C(j,i) X_{\alpha_j}.
\end{align*}

A set of such elements with the above commutation relations is said to be a canonical
generating set of $\g^\C$.

For a root $\alpha$ the corresponding reflection in the Weyl group is denoted
by $w_\alpha$. For $X\,\in\, \g^\C$, let $\ad \,X \,:\, \g^\C\,\longrightarrow\,
\g^\C$ be the adjoint
map, that is $\ad\, X (Y)\, =\, [X,\,Y] $. If $X\,\in\, \g^\C$ is such that $\ad \,X$ is
nilpotent, then $\exp( \ad\, X)$ is an automorphism of $\g$.
For $1\,\leq\, i\,\leq \,\ell$ define
$$\dot{w}_{\alpha_i} \,=\, \exp (\ad X_{\alpha_i}) \exp( -\ad X_{-\alpha_i})
\exp(\ad X_{\alpha_i})\, .$$
Then $\dot{w}_{\alpha_i}$ is an automorphism of $\g^\C$ with
$$\dot{w}_{\alpha_i} (\g_\alpha^\C) \,=\, \g_{w_{\alpha_i}(\alpha)}^\C$$ (see, e.g.,
\cite[Chapter 3, Lemma 19(b)]{St}).

Let $X\,\in\, \g^\C$ be nilpotent. Then by an elementary calculation (see for example
\cite[Lemma 2.3.1]{deG2}) it is seen that for any $Y\,\in\, \g^\C$ we have
$$\rho^\C (\exp( \ad \,X)(Y)) \,= \,\exp( \rho^\C(X) ) \rho^\C(Y) \exp( \rho^\C(X) )^{-1}\, .$$
So if we define
$$\dot{w}_{\alpha_i,\rho} \,=\, \exp (\rho^\C(X_{\alpha_i})) \exp( -\rho^\C( X_{-\alpha_i}))
\exp(\rho^\C( X_{\alpha_i}))\, ,$$
then $\dot{w}_{\alpha_i,\rho}\,\in\, G^\C_\rho$, and 
for $Y\,\in\, \g^\C$ we have $\rho^\C(\dot{w}_{\alpha_i}(Y))\, =\,
\dot{w}_{\alpha_i,\rho} \rho^\C(Y) \dot{w}_{\alpha_i,\rho}^{-1}$.

The Weyl group acts on the dual vector space $(\cc^\C)^*$, and also in a compatible
way on $\cc^\C$. Here we do not go in this direction, but only refer to
\cite[Remark 2.9.9]{deG2}. For $\mu\,\in\, (\cc^\C)^*$ and $H\,\in\, \cc^\C$ we have
that $w_\alpha(\mu)(H)\, =\, \mu(w_\alpha(H))$. Furthermore, for $H\,\in\, \cc^\C$ it
is well-known (see for example \cite[Lemma 8.4.1]{deG3}) that
$\dot{w}_{\alpha_i}(H) \,=\, w_{\alpha_i}(H)$, so that
\begin{equation}\label{eqw}
w_{\alpha_i}(\mu)(H)\, =\, \mu (\dot{w}_{\alpha_i}(H))\, .
\end{equation}

\subsection{Implementing the algorithms}\label{sec:comp}

For our algorithms we assume that a Lie algebra $\g$ is given by a basis and
a multiplication table, as this is the most convenient way to specify a
finite-dimensional Lie algebra in a way susceptible to computations
(see \cite{deG3}). Elements of $\g$ are given as coefficient vectors with respect
to the input basis. A representation $\rho \,:\, \g\,\longrightarrow\,\gl(V)$ is given by a
function that computes $\rho(X)v$ for given $X\,\in\,\g$, $v\,\in\, V$.

We have implemented the algorithm in the computer algebra system {\sf GAP}4
(\cite{gap4}) using the package {\sf CoReLG} (\cite{corelg}) which contains
functionality for working with real forms of complex simple Lie algebras.
These real forms in the package are given by a basis along with a
multiplication table. The structure constants in these multiplication
tables lie in $\Q$. The package {\sf CoReLG} also contains algorithms for
finding the Cartan subalgebras of $\g$ (up to conjugacy by the adjoint group).
It turns out that if we work with the real forms as given by the system
then the Cartan subalgebras are split over the field $\Q(\sqq)$. So in order to
compute roots, weights and weight spaces, we work with that field.

\section{Complex representations of complex semisimple Lie algebras}\label{se3}

In this section we briefly review the theory of weights of representations
of complex semisimple Lie algebras (see \cite[\S~7.3]{HN}, \cite[Chapter V]{Kn}
for details).
We also describe the straightforward algorithm to decompose a representation
of a complex semisimple Lie algebra as a direct sum of irreducible summands.

Let $\g^\C$ be a complex semisimple Lie algebra and $\rho^\C \,:\, \g^\C
\,\longrightarrow\,\gl(V^\C)$ a finite-dimensional complex representation.

Let $\cc^\C$ be a Cartan subalgebra of $\g^\C$. We consider the corresponding
root system. A positive system of roots, for example one
determined by fixing a basis of $\mathfrak{c}^\C$, determines a
Borel subalgebra $\mathfrak{b}^\C\,=\, \mathfrak{c}^\C+\mathfrak{y}^\C$ of $\g^\C$,
where $\mathfrak{y}^\C$ is spanned by all root vectors corresponding to
the chosen system of positive roots. In the sequel we denote the space
spanned by all root vectors corresponding to negative roots by
$\mathfrak{y}^\C_-$.

A linear functional $\lambda \,:\, \cc^\C \,\longrightarrow\, \C$ is called a weight of
$V$ if the space
$$V_\lambda^\C \,=\, \{ v\,\in\, V^\C\, \mid\, H\cdot v\, =\, \lambda(H)v\ \text{ for all }
H\,\in\, \cc^\C\}$$
is nonzero. In that case $V_\lambda^\C$ is called the weight space corresponding to
$\lambda$. The weight $\lambda$ is called a {\em highest weight} if
$N\cdot v \,= \,0$ for all $N\,\in \,\mathfrak{y}^\C$ and $v\,\in\, V_\lambda^\C$. In that case
a nonzero $v\,\in\, V_\lambda^\C$ is called a {\em highest weight vector}.
The next theorem is well known, see for example \cite[Theorem 7.3.6]{HN}.

\begin{theorem}\label{thm1}\mbox{}
\begin{enumerate}
\item If $\rho^\C$ is irreducible, then there is a unique highest weight
$\lambda$, and $\dim V_\lambda^\C \,=\, 1$. Moreover, if $v_\lambda$ denotes a
fixed nonzero element of $V_\lambda^\C$, then $V$ is spanned by elements
$Y_1\cdots Y_k \cdot v_\lambda$, where $k\,\geq\, 0$ and $Y_i\,\in\, 
\mathfrak{y}^\C_-$.

\item If $\lambda$ is a highest weight of $V$, and $v\,\in\, V_\lambda^\C$ is
nonzero, then the $\g^\C$-submodule of $V^\C$ generated by $v$ is
irreducible. 
\end{enumerate}
\end{theorem}

Theorem \ref{thm1} implies the following corollary (see \cite[Theorem 2.2.14]{CS} for a
similar statement).

\begin{corollary}\label{cor1}
Let $U\, =\, \{ v\,\in\, V^\C \,\mid\, N\cdot v \,=\, 0\ \text{ for all } N\,\in
\,\mathfrak{y}^\C\}$.
Let $u_1,\,\cdots,\,u_s$ be a basis of $U$ consisting of weight vectors. Let
$V_i^\C$ denote the $\g^\C$-submodule of $V^\C$ generated by $u_i$. Then
each $V_i^\C$ is irreducible and $V^\C$ is their direct sum.
\end{corollary}

This immediately yields an algorithm for the decomposition of $V^\C$. It consists
of the following steps:
\begin{enumerate}
\item Compute a basis of the space $U$. (This boils down to solving a set of
linear equations.)

\item $U$ is stable under $\cc^\C$; by diagonalizing the action of $\cc^\C$ on
$U$ compute a basis $u_1,\,\cdots, \,u_s$ of $U$ consisting of weight vectors.

\item Return the $\g^\C$-modules generated by the $u_i$ for $1\,\leq\, i\,\leq\, s$.
\end{enumerate}

\begin{remark}\label{rem:basis}
As in Corollary \ref{cor1} let $V_i^\C$ denote the $\g^\C$-module generated by
$u_i$. By Theorem \ref{thm1} a basis of $V_i^\C$ can by found by computing
enough elements of the form $Y_1\cdots Y_k\cdot u_i$, where $Y_i$ are basis
elements of $\mathfrak{y}^\C_-$, and taking a linearly independent set.
However, in the obvious algorithm for that one needs to compute many elements
of that form, and test linear dependence many times. There is an approach,
based on Littelmann's path model and the theory of canonical bases,
that yields a basis of $V_i^\C$ without performing any checks of linear
independence. Moreover, each basis element is obtained by acting with one
$Y_i$ on a basis element obtained previously. It is beyond the scope of this
paper to describe this algorithm in detail; it has been indicated in
\cite[Section 5.2.1]{deG2}, where an algorithm is given for computing an
admissible lattice of an irreducible module, using exactly the same approach.
\end{remark}

\section{Real representations of real semisimple Lie algebras}\label{se4}

We now consider a real representation $\rho \,:\, \g\,\longrightarrow\, \gl(V)$ of the real
semisimple Lie
algebra $\g$. We use the notation introduced in Section \ref{se2n}.
The aim of this section is to describe an algorithm for finding irreducible
submodules of $V$ such that $V$ is their direct sum. We have divided the
approach in
several subsections. In the first subsection we consider an involution
of the set of roots, and define three related objects: an element of
the Weyl group, an automorphism of $\g^\C$ and an element of $G_\rho^\C$.
The second subsection has the main algorithm as well as the results that
underpin it. The third section is devoted to an elaborate example illustrating
the algorithm. 
In the fourth subsection we briefly discuss our implementation and give
some run-times of our algorithm on a few sample inputs. 

We use a fixed Cartan subalgebra $\cc$ of $\g$. While this can
be any Cartan subalgebra of $\g$, in our implementation we use a
maximally non-compact Cartan subalgebra --- in case the
algebra is not compact (see Section \ref{sec:impl}). 

The complexification $\cc^\C$ is a Cartan subalgebra of $\g^\C$. 
Let $R$ be the set of roots of $\cc^\C$ in $\g^\C$, and let $R^+$ be a fixed
positive
system (for example one which is determined by fixing a basis of
$\mathfrak{c}$ as explained in Section \ref{se2n}). We let $\y^\C$
be the subalgebra of $\g^\C$ spanned by all $\g^\C_\alpha$ for $\alpha\,\in \,R^+$.

\subsection{An involution on the roots}\label{sec:invol}

Let $H\,\in\, \cc^\C$ and $X\,\in\, \g^\C_\alpha$; then
$$[H,\,\sigma(X)] \,= \,\sigma([\sigma(H),\,X]) \,=\, \overline{\alpha(\sigma(H))}
\sigma(X)\, .$$
Hence $\sigma(X)$ lies in the root space corresponding to the root
$H\,\longmapsto \,\overline{\alpha(\sigma(H))}$. We denote this root by $\alpha^\sigma$.
Then $\alpha\,\longmapsto \,\alpha^\sigma$ is an involution of $R$. Also we see that
$\sigma(\y^\C)$ is spanned by all $\g^\C_{\alpha^\sigma}$ for $\alpha\,\in\, R^+$.

In general, if $\lambda \,:\, \cc^\C\,\longrightarrow\,\C$ is linear then by
$\lambda^\sigma$ we denote the linear function $H\,\longmapsto\,
\overline{\lambda(\sigma(H))}$.

\begin{lemma}\label{lem1}
Let $p \,:\, R\,\longrightarrow\, R$ be defined by $p(\alpha) \,=\, \alpha^\sigma$. Then
there are simple roots $\beta_1,\,\cdots,\,\beta_l$ such that $w_{\beta_1}\cdots w_{\beta_l}
(R^+) \,=\, p(R^+)$.
\end{lemma}

\begin{proof}
If $\sigma(\mathfrak{y})$ is not equal to $\mathfrak{y}$, there must
be a simple positive root $\alpha$ such that ${\alpha}^{\sigma}$ is
negative.

List the positive roots which are mapped to negative roots by the map $p$ as $${\alpha
}_1\,=\,\alpha ,\, {\alpha}_2,\,\cdots ,\,{\alpha}_k\, ,$$ and let ${\alpha}_{k+1},\, \cdots
,\, {\alpha}_N$ be the remaining positive roots, so that $p({\alpha}_{k+1}),\, \cdots
,\, {p(\alpha}_N)$ are all positive.

Now $w_{{\alpha}_1}$ maps ${\alpha}_1$ to $-{\alpha}_1$ and permutes the remaining
positive roots, that is $$w_{{\alpha}_1}(R^+)\,=\, \{-\alpha_1,\,
\alpha_2,\,\cdots ,\, \alpha_k,\, \alpha_{k+1},\, \cdots ,\, {\alpha}_N\}\, .$$
Therefore $$p(w_{{\alpha}_1}(R^+))\, =\, \{-p({\alpha}_1),\,
p(\alpha_2),\,\cdots ,\,p(\alpha_k),\, p(\alpha_{k+1}),\, \cdots ,\,p(\alpha_N)
\}\, .$$ In this list, only $p(\alpha_2),\,\cdots ,\,p(\alpha_k)$ are
negative. Thus the number of positive roots mapped to negative roots by
$p\circ w_{\alpha_1}$ decreases by 1.

Replacing $p$ by $p\circ w_{{\alpha}_1}$ and repeating the argument we see that there
are simple roots ${\beta}_1\,=\,{\alpha}_1,\,\cdots, \, \beta_l$ such that
$p(w_{{\beta}_1}\ldots w_{\beta_l}(R^+))\,=\, R^+$, and therefore $w_{{\beta}_1}\ldots
w_{{\beta}_l}(R^+)\,=\,p(R^+)$ as $p$ is an involution.
\end{proof}

\begin{remark}\label{rem:findwi}
Note that the above proof gives an immediate algorithm for finding
$$\beta_1,\,\cdots,\,\beta_l\, .$$ Indeed, $\beta_1$ is a simple root mapped to a
negative root by $p$. Then $\beta_2$ is a simple root mapped to a negative
root by $p\circ w_{\beta_1}$ and so on. The algorithm stops when we have
found $\beta_1,\,\cdots,\,\beta_l$ such that $pw_{\beta_1}\cdots w_{\beta_l}(R^+)
\,=\,R^+$. 
\end{remark}

In the sequel we use the following notation.
Let $\beta_1,\,\cdots,\,\beta_l$ be as in Lemma \ref{lem1}. Then set
\begin{align*}
w & \,=\, w_{\beta_1}\cdots w_{\beta_l}\\
\omega &\,=\, \dot{w}_{\beta_1}\cdots \dot{w}_{\beta_l}\\
\omega_\rho &\,=\, \dot{w}_{\beta_1,\rho}\cdots \dot{w}_{\beta_l,\rho}\, .
\end{align*}

In view of what was seen in Section \ref{se2n} we now immediately
have the following:

\begin{corollary}\label{cor2}
The above defined $\omega$ is an automorphism of $\g^\C$ with $\omega(\y) \,=\, \sigma(\y)$.
Secondly, $\omega_\rho$ is an element of $G_\rho^\C$ with
$$\rho^\C( \omega(Y) )\,=\, \omega_\rho \rho^\C(Y) \omega_\rho^{-1}$$
for all $Y\,\in\, \g^\C$. 
\end{corollary}

As in Section \ref{se3} we consider the weights of $V^\C$ with respect to the
fixed Cartan subalgebra $\cc^\C$. We recall that a weight $\lambda$ is
said to be a highest weight if $N\cdot v \,=\, 0$ for all $N\,\in\, \y^\C$ and
$v\,\in\, V^\C_\lambda$.

\begin{lemma}\label{lem3}
Let $v\,\in\, V^\C$ be a highest weight vector with highest weight $\lambda$.
Then $\omega_\rho^{-1}\overline{v}$ is a highest weight vector with highest
weight $w^{-1} \lambda^\sigma$.
\end{lemma}

\begin{proof}
By definition $\rho^{\mathbb{C}}(\mathfrak{y})\cdot v\,=\,0$, and therefore
$\rho^{\mathbb{C}}(\overline{\mathfrak{y}})\cdot \overline{v}\,=\,0$. Hence by
Corollary \ref{cor2},
$$\rho^{\mathbb{C}}(\omega (\mathfrak{y}))\cdot \overline{v}\,=\,0\, .$$
Again by Corollary \ref{cor2} we have 
$\rho^{\mathbb{C}}(\omega (Y))\,=\,{\omega}_{\rho}
\rho^{\mathbb{C}}(Y)\omega^{-1}_{\rho}$, so that
$$
\rho^{\mathbb{C}}(\mathfrak{y})\omega^{-1}_\rho \overline{v}\, =\, 0\, .
$$
Consequently $\omega_\rho^{-1}\overline{v}$ is a highest weight vector.

Now, for all $H\,\in\, \mathfrak{c}^{\mathbb{C}}$ we have $\rho^\C(H)v\,=\,
\lambda(H)v$, which implies that $$\rho^\C(H)\overline{v}\,=\,\lambda^{\sigma}(H)
\overline{v}\, .$$ Therefore,
\begin{multline*}
\rho^\C(H) \omega^{-1}_\rho \overline{v}\,=\,
\omega^{-1}_\rho \omega_\rho\, \rho^\C(H)\omega^{-1}_{\rho}\overline{v}
\,=\, \omega^{-1}_\rho\,\rho^\C(\omega(H))\overline{v}\\
=\, \omega^{-1}_\rho \lambda^\sigma(\omega(H))\overline{v}\,=\,
\lambda^\sigma(\omega (H))(\omega^{-1}_{\rho}\overline{v})
\,=\, (w^{-1}\lambda^\sigma (H))(\omega^{-1}_\rho\,\overline{v})\, .
\end{multline*}
Here the last equality follows from \eqref{eqw}. This establishes the last
assertion in the lemma.
\end{proof}

In the sequel for a weight $\mu\in(\cc^\C)^*$
we write
\begin{equation}\label{th}
\Theta(\mu) \,=\, w^{-1} \mu^\sigma\, .
\end{equation}

\begin{lemma}\label{lem4}
The above function $\Theta$ is an involution.
\end{lemma}

\begin{proof}
By \eqref{eqw}, for $H\,\in\, \cc^\C$ we have $w^{-1}\mu^\sigma(H) \,=\,
\mu^\sigma( \omega(H))$. The latter is equal to
$\overline{\mu( \sigma\omega(H) )}$, so that 
$$\Theta^2(\mu) (H) \,=\, \mu( \sigma\omega\sigma\omega (H) ).$$
On the one hand it is straightforward to see that $\sigma \exp( \ad\, X ) \,=\,
\exp( \ad\, \sigma(X) )\sigma$ for any $X\,\in\, \g^\C$. Hence $\sigma \omega \,=\,
\widehat{\omega} \sigma$, where $\widehat\omega$ is an inner automorphism of $\g^\C$.
So $\sigma\omega\sigma\omega\, =\, \widehat{\omega}\omega$ is an inner automorphism
of $\g^\C$. It stabilizes $\cc^\C$ therefore its restriction to $\cc^\C$ is an
element $v$ of the Weyl group.

On the other hand we have $\omega (\y) \,=\, \sigma(\y)$
(Corollary \ref{cor2}) so that $\sigma\omega\sigma\omega( \y) \,=\, \y$. This
implies that $v$ maps positive roots to positive roots, which in turn implies
that $v$ is the identity. The conclusion is that $\sigma\omega\sigma\omega(H)
=H$ for all $H\in \cc^\C$, showing that $\Theta^2(\mu) \,=\, \mu$.
\end{proof}

\subsection{Decomposing a real representation}

Here we describe our algorithm. One of the main points is the next lemma,
relating the map $\Theta$ to the representation theory of $\g$.

In \cite[\S~7, Equation (10)]{On} an involution of the set of
weights in $(\cc^\C)^*$ is defined; this involution is denoted $s_0$. Here we do not
give the details of this definition, but refer to \cite{On} for the main properties
of $s_0$. For us the following is relevant.

\begin{lemma}
For $\Theta$ in \eqref{th} the equality $$\Theta\,=\, s_0$$ holds.
\end{lemma}

\begin{proof}
Let $\lambda$ be a dominant weight. We will show that $\Theta(\lambda)\,=\,s_0(\lambda)$;
this would prove the lemma.
Let $V^\C$ be the irreducible $\g^\C$-module of highest weight $\lambda$. Let
$V_\R^\C$ denote its realification. Let $\rho$ denote the corresponding
representation of $\g$. Consider the complexification $(V^\C_\R)^\C$ of $V^{\C}_{\R}$,
which is a $\g^\C$-module. By combining \cite[\S~8, Proposition 1]{On}
and \cite[\S~8, Theorem 3]{On} we see that $(V^\C_\R)^\C$ is the direct sum of
two irreducible $\g^\C$-modules $U^\C$ and $U_0^\C$ of highest weights $\lambda$ and
$s_0(\lambda)$ respectively.
So if $s_0(\lambda)\,=\,\lambda$ then the only highest weight occurring in
$(V^\C_\R)^\C$ is $\lambda$. Hence using Lemma \ref{lem3} it follows that 
$\Theta(\lambda)\,=\,\lambda$ as well. If $s_0(\lambda)\,\neq\, \lambda$, then the
direct sum decomposition of the $\g^\C$-module $(V^\C_\R)^\C$ is unique.
Let $v$ be a highest-weight vector of $U^\C$. If $\omega_\rho^{-1} \overline{v}$ is a
multiple of $v$, then $U^\C \,=\, \overline{U^\C}$ so that
$$M\,=\,\{ u\,\in\, U^\C \,\mid\, u \,= \,\overline{u}\}$$
forms a nontrivial irreducible $\g$-submodule of $V^\C_\R$. But from the fact that
$s_0(\lambda)\,\neq\, \lambda$ and \cite[\S~8, page 69, point 3]{On}
it follows that $V^\C_\R$ is irreducible as $\g$-module. Therefore, we have
$M\,=\,V_\R^\C$. But that cannot happen as $\dim_\R M = \dim_\C U^\C= \tfrac{1}{2}
\dim_\R V_\R^\C$. Hence $\omega_\rho^{-1} \overline{v}$ is a highest weight vector of
$U_0^\C$. But that implies that $s_0(\lambda)\,=\,\Theta(\lambda)$ by Lemma
\ref{lem3}.
\end{proof}

In what follows we just use the symbol $\Theta$, and not $s_0$; this is also followed while
referring to results of \cite{On}.

Let $\lambda$ be a dominant weight, and let $\epsilon(\lambda)$ be its
Cartan index \cite[\S~8]{On} (there this index is denoted
$\epsilon(\g_0,\rho_0)$, where $\g_0$ is the real semisimple Lie algebra,
and $\rho_0$ is the representation corresponding to $\lambda$). Let $V^\C$ be an
irreducible $\g^\C$-module of highest weight $\lambda$.
Now $\lambda$ determines an irreducible $\g$-module in precisely one of the
following ways:
\begin{enumerate}
\item \textbf{(Type I)}~
Real case: $\Theta(\lambda)\,=\,\lambda$, $\epsilon(\lambda)\,=\,1$. Then $V^\C$ has
a basis whose real span is invariant under $\g$. This real span is the
irreducible module.

\item \textbf{(Type II)}~
Quaternion case: $\Theta(\lambda) \,=\, \lambda$, $\epsilon(\lambda)\,=\,-1$.
The realification $V_\R^\C$ is the irreducible module determined by
$\lambda$.

\item \textbf{(Type III)}~ Complex case: $\Theta(\lambda)\,\neq\, \lambda$.
The realification $V_\R^\C$ is the irreducible module determined by
$\lambda$.
\end{enumerate}

Conversely, every irreducible $\g$-module arises in one of the above three ways.

Note that types II and III are different in the following way: in type II
we have that $(V_\R^\C)^\C$ is the direct sum of two irreducible $\g^\C$-modules
of the same highest weight, whereas in type III, $(V_\R^\C)^\C$ is the direct
sum of two irreducible $\g^\C$-modules of different highest weights.

The next theorem is the basis of our algorithm.

\begin{theorem}\label{thm:des}
Let $\rho \,:\, \g\,\longrightarrow\, \gl(V)$ be a finite-dimensional representation of 
$\g$. Let $\mu$ be the highest weight
of an irreducible summand of the $\g^\C$-module $V^\C$. Let $V^\C(\mu)$ be the
space of highest-weight vectors in $V^\C$ of weight $\mu$. Then we have the
following cases.
\begin{enumerate}
\item If $\Theta(\mu)\,=\,\mu$, $\epsilon(\mu)\,=\,1$, then $V^\C(\mu)$ has a
basis $u_1,\,\cdots,\,u_s$ with the following property. Let $M_i$ denote the
irreducible $\g^\C$-submodule of $V^\C$ generated by $u_i$. Then $M_i \,=\,
\overline{M}_i$ and the real points in $M_i$ form an irreducible summand
of $V$.

\item If $\Theta(\mu)\,=\,\mu$, $\epsilon(\mu)\,=\,-1$, then $\dim V^\C(\mu)$
is even and $V^\C(\mu)$ has a basis consisting of
$u_1,\,\cdots,\,u_s$, $v_1,\,\cdots,\,v_s$ with the
following property. Let $M_i$ and $N_i$ denote the
irreducible $\g^\C$-submodules of $V^\C$ generated by $u_i$ and $v_i$
respectively. Then $N_i \,=\, \overline{M}_i$, and the real points of
$M_i\oplus N_i$ form an irreducible summand of $V$.

\item If $\Theta(\mu)\,\neq\, \mu$, then setting $\nu \,=\, \Theta(\mu)$ we have
that $\nu$ is also a highest weight of $V^\C$. Moreover, there are bases
$u_1,\,\cdots,\,u_s$ of $V^\C(\mu)$ and $v_1,\,\cdots,\,v_s$ of $V^\C(\nu)$
with the following property. Let $M_i$ and $N_i$ denote the
irreducible $\g^\C$-submodules of $V^\C$ generated by $u_i$ and $v_i$
respectively. Then $N_i \,=\, \overline{M}_i$ and the real points of
$M_i\oplus N_i$ form an irreducible summand of $V$.
\end{enumerate}
\end{theorem}

\begin{proof}
The proof shows how to find these bases. We use the following key result
(Lemma \ref{lem3}):
let $v\,\in\, V^\C$ be a highest-weight vector of weight $\mu$, and let $M$ be
the irreducible $\g^\C$-module generated by $v$. Then
$\omega_\rho^{-1}\overline{v}$ is a highest-weight vector of weight $\Theta(\mu)$,
and the irreducible module generated by it is $\overline{M}$.

Suppose that we are in the first case. We show how to obtain a finite set of
elements $S$ of $V^\C(\mu)$ such that
\begin{itemize}
\item $S$ spans $V^\C(\mu)$, and

\item for each $u\,\in\, S$ the $\g^\C$-module $M$ generated by $u$ satisfies
$M\,=\,\overline{M}$.
\end{itemize}

Let $u\,\in\, V^\C(\mu)$ be an element of a fixed basis of $V^\C(\mu)$, and set
$v\,=\,\omega_\rho^{-1}\overline{u}$. Let $M$ and $N$ be the $\g^\C$-modules generated by
$u$ and $v$ respectively. If $v$ is a scalar multiple of $u$ then $M\,=\,N$ so that
$\overline{M} \,=\, M$ and we add $u$ to $S$. Otherwise set $W\,=\,M\oplus N$; we have
that $\overline{W} \,=\, W$ so the real points $W(\R)$ in $W$ form a $\g$-module.
The only highest-weight occurring in $W(\R)^\C$ is $\mu$ so $W(\R)\,=\,W_1\oplus W_2$,
where $W_i$ is an irreducible $\g$-module of type I corresponding to the
highest weight $\mu$. It follows that there is a $\delta\,\in\, \C$ such that
$u+\delta v$ is a highest weight vector of $W_1^\C$. Since $W_1^\C$ is self
conjugate we have that $\omega_\rho^{-1} \overline{u+\delta v}$ is a scalar
multiple of $u+\delta v$. Observe also that $\omega_\rho^{-1} \overline{v} \,=\, \gamma u$
for some $\gamma\,\in\, \C$. Hence
$$\omega_\rho^{-1} \overline{u+\delta v}\, =\, v + \overline{\delta} \gamma u\, .$$
This being a scalar multiple of $u+\delta v$ we conclude that $\gamma \,=\, \tfrac{1}
{|\delta|^2}$, so that $\gamma$ is a positive real number. Then it follows
that $$u_{\pm}\,=\, u \pm \tfrac{1}{\sqrt{\gamma}} v$$ are highest weight vectors
with $\omega_\rho^{-1}(\overline{u}_{\pm}) \,=\, \pm\sqrt{\gamma} u_{\pm}$.
So we add the elements $u_{\pm}$ to the set $S$.

We do this for all elements of a basis of $V^\C(\mu)$. Then from $S$ we select a
maximal linearly independent subset, and we are done in the first case.

Now consider the second case. Here an irreducible module corresponding to $\mu$
is of type II. So for a nonzero $u\,\in\, V^\C(\mu)$ we cannot have that
$\omega_\rho^{-1}\overline u$ is a scalar multiple of $u$.

Suppose that we have
constructed $u_1,\,\cdots,\,u_k$, $v_1,\,\cdots,\,v_k$. Let $u\,\in \,V^\C(\mu)$ not lie in
the span of these elements. Then set $u_{k+1} \,=\, u$, $v_{k+1} \,=\, \omega_\rho^{-1}
\overline{u}$. We show that $u_1,\,\cdots,\,u_{k+1}$, $v_1,\,\cdots,\,v_{k+1}$ are linearly
independent. Let $M_i,N_i$ be as in the theorem.
Let $A$ denote the sum of all $M_i$, $N_i$ for $1\,\leq\, i\,\leq\, k$. Then
$A$ is self-conjugate. Let $$B\,=\,M_{k+1}\oplus N_{k+1}\, .$$ Then $B$ is self-conjugate
as well. Therefore, the intersection of $A$ and $B$ is self-conjugate.
So if $A\cap B\,\neq\, 0$, then it contains a nonzero real point. But the
real points of $B$ form an irreducible $\g$-module. Consequently, $B\,\subset\, A$,
which is not the case. Hence $A\cap B\,=\,0$ which completes the proof.

The proof of the third case is similar. (In fact it is even more straightforward.)
\end{proof}

\subsection{The algorithm}\label{sec:algorithm}

Below we give a
the steps of the algorithm in detail. We use three subroutines, which are
given below the main algorithm. Their inclusion into the main algorithm
would obscure the structure of the latter. Therefore we give them separately.

The algorithm follows directly from Theorem \ref{thm:des}. The proof of the
theorem also proves the correctness of the algorithm.

The input is a real semisimple Lie algebra $\g$ over $\R$ together with a
finite-dimensional representation
$$\rho\,:\, \g \,\longrightarrow\, \gl(V)\, .$$
The output is a set of irreducible submodules of $V$ such that $V$ is their
direct sum.

The algorithm takes the following steps:
\begin{enumerate}
\item Compute a Cartan subalgebra $\cc$ of $\g$.

\item Compute the root system $R$ of $\g^\C$ with respect to $\cc^\C$.

\item Compute a set of simple roots $\{\alpha_1,\,\cdots,\,\alpha_\ell\}$ of
$R$ and for $1\leq j\leq \ell$ let $X_{\alpha_j}$ be a fixed nonzero
element of the root space $\g^\C_{\alpha_j}$.

\item Compute the highest weights $\mu_1,\,\cdots,\,\mu_s$ of the
complexification $V^\C$ of $V$ (see Section \ref{se3}).

\item For $1\leq i\leq s$ compute a basis of the space
$$V^\C(\mu_i) \,=\, \{ v\,\in\, V^\C \,\mid\, \rho^\C(H)v = \mu_i(h) v\ \text{ and }\
\rho^\C(X_{\alpha_i})v\, =\, 0 \ \text{ for }\ 1\,\leq\, j\,\leq\, \ell\}$$
(which is the space of highest weight vectors of weight $\mu_i$).

\item Set $Q \,=\, \emptyset$. For $i\,=\,1,\,\cdots,\,s$ do:
\begin{enumerate}
\item[(a)] If $\Theta(\mu_i)\,=\,\mu_i$ and $\epsilon(\mu_i)\,=\,1$, then
set $Q \,:=\, Q\cup {\tt SubmodulesI}(\mu_i,V^\C(\mu_i))$. 

\item[(b)] If $\Theta(\mu_i)\,=\,\mu_i$ and $\epsilon(\mu_i)\,=\,-1$, then
set $Q \,:=\, Q\cup {\tt SubmodulesII}(\mu_i,V^\C(\mu_i))$.

\item[(c)] If $\Theta(\mu_i)\,\neq \,\mu_i$, then set $\nu \,= \,\Theta(\mu_i)$.
If $\nu\,=\,\mu_j$ with $j\,>\,i$, then 
$Q \,:= \,Q\cup {\tt SubmodulesIII}(\mu_i,V^\C(\mu_i))$. If $j\,<\,i$ then
in this step we do nothing.
\end{enumerate}
\item Return $Q$.
\end{enumerate}

Now we describe the subroutines used in the algorithm. They correspond exactly
to the three cases of Theorem \ref{thm:des}.

{\tt SubmodulesI}$(\mu_i,V^\C(\mu_i))$ takes the following steps:
\begin{enumerate}
\item Set $S\,=\,\emptyset$, and for $u$ in a basis of $V^\C(\mu_i)$ we do following:
\begin{enumerate}
\item Set $v\,=\,\omega_\rho^{-1} \overline{u}$.

\item If $v$ is a scalar multiple of $u$ then add $u$ to $S$.

\item Otherwise compute $\gamma\,\in \,\C$ such that $\omega_\rho^{-1} \overline{v}\,=\,\gamma u$.
(In the proof of Theorem \ref{thm:des} we have seen that $\gamma\,\in\, \R$ and
$\gamma \,>\,0$). Add the elements $u\pm \tfrac{1}{\sqrt{\gamma}} v$ to $S$.
\end{enumerate}
\item Compute a maximal linearly independent subset $\{u_1,\,\cdots,\,u_t\}$ of $S$.

\item For any $1\,\leq\, k\,\leq\, t$ let $W_k^\C$ be the $\g^\C$-submodule of $V^\C$
generated by $u_k$. Let $W_k$ be the module of real points of $W^\C$.

\item Return $W_1,\,\cdots,\,W_t$.
\end{enumerate}

{\tt SubmodulesII}$(\mu_i,V^\C(\mu_i))$ takes the following steps:
\begin{enumerate}
\item Set $A\,=\,B\,=\,\emptyset$. While $|A|+|B| \,<\, \dim V^\C(\mu_i)$ do the following:
\begin{enumerate}
\item Let $u\,\in\, V^\C(\mu_i)$ lie outside the span of $A\cup B$.
\item Set $A\,:=\, A\cup \{u\}$, $B\,:=\, B\cup\{ \omega_\rho^{-1}\overline{u}\}$.
\end{enumerate}
\item Write $A\,=\,\{ u_1,\,\cdots,\,u_t\}$. For $1\,\leq\, k\,\leq \,t$, let
$M_k$ and $N_k$ be the $\g^\C$-submodules of $V^\C$ generated by
$u_k$ and $\omega_\rho^{-1} \overline{u}_k$ respectively. Let $W_k$ be the $\g$-module
consisting of the real points of $M_k\oplus N_k$.
\item Return $W_1,\,\cdots,\,W_t$.
\end{enumerate}

{\tt SubmodulesIII}$(\mu_i,V^\C(\mu_i))$ takes the following steps:
\begin{enumerate}
\item Let $u_1,\,\cdots,\,u_t$ be a basis of $V^\C(\mu_i)$. 

\item For $1\,\leq\, k\,\leq\, t$, let
$M_k$ and $N_k$ be the $\g^\C$-submodules of $V^\C$ generated by
$u_k$ and $\omega_\rho^{-1} \overline{u}_k$ respectively. Let $W_k$ be the $\g$-module
consisting of the real points of $M_k\oplus N_k$.

\item Return $W_1,\,\cdots,\,W_t$.
\end{enumerate}

\begin{remark}
In the above algorithm we do not need to compute $\epsilon(\mu_i)$. Instead we may
just compute $v \,=\, \omega_\rho^{-1}\overline{u}$ for a nonzero $u\,\in\, V^\C(\mu_i)$. If
$v$ is a multiple of $u$, then necessarily $\epsilon(\mu_i)\,=\,1$. Otherwise
we compute $\gamma\,\in \,\C$ such that $\omega_\rho^{-1}\overline{v}\,=\,\gamma u$.
If $\gamma$ is a positive real number, then again $\epsilon(\mu_i)\,=\,1$, otherwise
it is $-1$.
\end{remark}

\begin{remark}
The irreducible $\g$-modules of type II and type III are easily constructed.
The algorithm given here can be used to construct the modules of
type I as well. This works as follows. Let $\lambda$ be a dominant weight with
$\Theta(\lambda)\,=\,\lambda$ and $\epsilon(\lambda)\,=\,1$. Then construct
the irreducible $\g^\C$-module $V^\C$ with highest weight $\lambda$. Let
$V_\R^\C$ be its realification. Then the $\g$-module $V_\R^\C$ splits as
$V_\R^\C \,=\,
W_1\oplus W_2$, where $W_1,\, W_2$ are isomorphic irreducible $\g$-modules of
type I.
\end{remark}

\begin{remark}\label{rem:symbolic}
In some settings the $\g$-module $V$ is not given explicitly. Here we
consider two such examples. First let $W$ be an irreducible $\g$-module
of type I, II or III, just
given by its highest weight, and set $V\,=\,W\otimes W$. A second example occurs
when $\g$ is a subalgebra of a semisimple subalgebra $\a$ and $V$ is
an irreducible $\a$-module given by its highest weight. We consider $V$
as $\g$-module (in this case the decomposition of $V$ is called
a branching rule). In situations like these we can also find the decomposition
of $V$. Since we do not have $V$ explicitly we just want the types and
highest weights of the direct summands of $V$.

As before we let $\cc$ be a Cartan subalgebra of $\g$.
Then we first find the weights of $V^\C$ (with respect to
$\cc^\C$) and their multiplicities (that is, the dimensions of the
corresponding weight spaces). If $V\,=\,W\otimes W$ as above then we can determine
the weights of $W^\C$ and their multiplicities by Freudenthal's formula
(see \cite[\S~8.8]{deG3}). From these we immediately get the weights
and multiplicities of $V^\C$. For the computation of a branching rule we
first compute a Cartan subalgebra $\h^\C$ of $\a^\C$ containing $\cc^\C$.
Again the weights (with respect to $\h^\C$) and multiplicities of the
$\a^\C$-module $V^\C$ can be computed using Freudenthal's formula. From that
the weights of $V^\C$ can readily be computed: given bases of $\h^\C$ and
$\cc^\C$ this is a simple projection; see \cite[\S~8.13]{deG3}. 

From the weights of $V^\C$ and their multiplicities it is straightforward to
find the highest weights of the irreducible summands of $V^\C$ (see
\cite[\S~8.12]{deG3} for a description of the algorithm for this purpose).
Let $\mu$ be such a highest weight, and let $m$ denote the number of summands
with highest weight $\mu$.
If $\Theta(\mu)\,=\,\mu$, then we compute the Cartan index $e\,=\,\epsilon(\mu)$
(see \cite[\S~8]{On}). If $e\,=\,1$, then $\mu$ is the highest weight of a module
of type I, occurring with multiplicity $m$. If $e\,=\,-1$, then $\mu$ is the
highest weight of an irreducible module of type II, occurring with
multiplicity $\tfrac{m}{2}$. If $\nu\,=\,\Theta(\mu)\,\neq\, \mu$, then $\mu$
determines an irreducible submodule of type III, occurring with multiplicity $m$.

We remark that in this setting, since the modules are not explicitly
given --- and consequently we cannot compute the element $\omega_\rho^{-1}$ --- we
need to compute the Cartan index $\epsilon(\mu)$ of a weight $\mu$.
Here we do not go into how this is done, but refer to \cite[\S~8]{On}.
\end{remark}

\subsection{An example}

We let a semisimple real matrix Lie algebra act on spaces of homogeneous
polynomials of a fixed degree. This works as follows.

Let $\g \,\subset\, \gl(n,\R)$ be a Lie subalgebra. Then $\g$ acts naturally on
the space $U\,=\,\R^n$. So it also acts on the dual space $U^*$ by
$(X\cdot f)(u) = -f(Xu)$. Let $v_1,\,\cdots,\,v_n$ be the natural basis of
$U$, and let $x_1,\,\cdots,\,x_n$ be dual the basis of $U^*$ defined by $x_i(v_j) \,=
\,\delta_{ij}$. Then $\g$ acts on the polynomial ring $\R[x_1,\,\cdots,\,x_n]$ by
$$ X\cdot f \,=\, \sum_{i=1}^n (X\cdot x_i) \partial_{x_i}(f).$$
It is clear that the space of homogeneous polynomials of degree $d$
is a $\g$-submodule; we ask what its irreducible summands are.

In our example we let $\g \,=\, \mathfrak{so}(4) \,=\,\{ X\in \gl(4,\R) \,\mid\,
X^T \,=\, -X \}$. A basis of $\g$ is
\begin{align*}
&e_1\,=\, e_{12}-e_{21}\, ,e_2\,=\, e_{13}-e_{31}\, ,e_3\,=\, e_{14}-e_{41}\, ,\nonumber \\
&e_4\,=\, e_{23}-e_{32}\, ,e_5\,=\, e_{24}-e_{42}\, , e_6\,=\, e_{34}-e_{43}.
\end{align*}

In the first three steps of the algorithm we compute a Cartan subalgebra $\cc$
of $\g$, its root system and a set of simple roots and corresponding root
vectors. 

In the Lie algebra at hand a Cartan subalgebra $\cc$ is spanned by $e_2,\,e_5$. 
The roots of $\cc$ are
$$
\alpha\,:=\, (-\sqrt{-1}\, , \sqrt{-1})\, ,\ \ \beta\,:=\, (\sqrt{-1}\, ,\ \
\sqrt{-1})\, ,\ -\alpha\, ,\ -\beta\, .
$$
Thus the root system is of type $A_{1} \times A_{1}$. Notice that conjugation maps every root to its negative. 

In order to describe the root vectors we define the following elements
\begin{align*}
& I_1\,=\, e_2-e_5\, ,\ J_1\,=\, e_3+e_4\, ,K_1\,=\, e_1+e_6\\
& I_2\,=\, e_2+e_5\, ,\, J_2\,=\, e_3-e_4\, , K_2\, =\, e_1-e_6.
\end{align*}
Then each triple $(I_i,J_i,K_i)$ spans a copy of $\mathfrak{so}(3)$ and $\g$
is their direct sum.

Now set
\begin{align*} 
&X_{1} = \frac{J_{1}-\sqrt{-1}K_{1}}{2},~ Y_{1} = \frac{J_{1}+\sqrt{-1}K_{1}}{2}\,, \\
&X_{2} = \frac{J_{2}-\sqrt{-1}K_{2}}{2},~ Y_{2} = \frac{J_{2}+\sqrt{-1}K_{2}}{2}\,.
\end{align*}
Then $X_1,X_2$ are the positive root vectors (as well as the root vectors
corresponding to the simple roots) and $Y_1,Y_2$ are the negative root vectors.
They satisfy the commutation relations
\begin{align*}
&[X_{1},\,Y_{1}]\,=\,-\sqrt{-1}I_{1},\ [I_{1},\,X_{1}]\,=\, -2\sqrt{-1} X_{1},\ [I_{1},\,Y_{1}]\,
=\, 2\sqrt{-1}Y_{1}\,,\\
&[X_{2},\,Y_{2}]\,=\,\sqrt{-1}I_{2},\ [I_{2},\,X_{2}]\,=\, 2\sqrt{-1} X_{2},\ [I_{2},
\,Y_{2}]\,=\, -2\sqrt{-1}Y_{2}\,.
\end{align*}
Hence the two triples $\sqrt{-1}I_1,\,X_1,\,-Y_1$ and $-\sqrt{-1}I_2,\,X_2,\,-Y_2$ satisfy
the commutation relations of $\mathfrak{sl}_2$. We define
$H_1\,=\, \sqrt{-1}I_1$, $H_2\,=\,-\sqrt{-1}I_2$. Let $V^\C$ be an irreducible
$\g^\C$-module with highest weight $\lambda$. Write $m_i\,=\, \lambda(H_i)$.
Suppose that $\Theta(\lambda)\,=\,\lambda$. Then
from \cite[\S 8 Proposition 3 and Table 5, p. 79]{On} it follows that
the Cartan index of $V^\C$ is $\epsilon(\lambda) \,=\, (-1)^{m_1+m_2}$. 

We now consider the action of $\g$ on the polynomial ring in four indeterminates
which we denote $x,\,y,\,z,\,w$ (instead of $x_1,\,x_2,\,x_3,\,x_4$). As in Steps 4, 5
we first find the highest weights and the corresponding highest weight spaces.

The vector fields
corresponding to $X_1$ and $X_2$ are respectively
\begin{align*}
&\frac{(\sqrt{-1}\,y-w)}{2}\partial_{x}-\frac{(z+\sqrt{-1}\,x)}{2}\partial_{y} +\frac{(\sqrt{-1}\,w+y)}{2} \partial_{z}+ \frac{(x-\sqrt{-1}\,z)}{2}\partial_{w}\,,\\
&\frac{(\sqrt{-1}\,y-w)}{2}\partial_{x}+\frac{(z-\sqrt{-1}\,x)}{2}\partial_{y} -\frac{(\sqrt{-1}\,w+y)}{2} \partial_{z} +\frac{(x+\sqrt{-1}\,z)}{2}\partial_{w}\,,
\end{align*}
and they are of rank two. Thus, they must have two functionally independent invariants. Calculating the joint invariants in linear and then quadratic 
polynomials (see \cite{ABGM} for algorithms to do that)
we find that these invariants are
$$e\,=\,w-\sqrt{-1}\,y \ \ \text{ and }\ \ f\,=\,x^2+z^2+2y(y+\sqrt{-1}\,w)\, .$$ As $de\bigwedge df$
is not identically zero, we know that $e$ and $f$ are functionally independent joint invariants of $X_{1}$ and $X_{2}$ that generate all other invariants. Keeping in mind that $e$ and $f$ are of degrees $1$ and $2$
respectively, we obtain, in the space of homogeneous polynomials of degree $d\,\geq\, 2$, the subspace $U$ of joint invariants 
$\langle e^d,e^{d-2}f,\,
e^{d-4}f^2,\cdots,f^{d/2}\rangle$
for $d$ even, and $\langle e^d,e^{d-2}f,\,
e^{d-4}f^2,\cdots, \, ef^{(d-1)/2}\rangle$ for $d$ odd.

The vector field corresponding to $I_1$ is 
\[
-z\partial_{x} + w\partial_{y} + x\partial_{z} -y \partial_{w}
\]
and $I_{1}\cdot e \,=\, -\sqrt{-1}\, e$,\ $I_{1}\cdot f \,=\, 2\sqrt{-1}\, e^2$,
so that $H_1\cdot e \,=\, e$ and $H_1\cdot f \,=\, -2e^2$. Therefore, 
$$H_{1}\cdot e^a f^b \,= \,a e^a f^{b}-2b\, e^{a+2} f^{b-1}.$$
The matrix of $H_{1}$ on $U$ is bidiagonal with eigen-values 
\[ d, d-2,d-4, \cdots \]
We find that the matrix of $H_2$ on $U$ is exactly the same as the matrix
of $H_1$ on $U$. 

Now we compute the highest weight vectors for the modules of homogeneous
polynomials in degrees $2,3,4$. This boils down to finding the eigenvectors of
$H_1$ (as the eigenvectors of $H_2$ are exactly the same). In degree 2 they are
$e^2, f+e^2$. In degree 3 they are $e^3,\, ef+e^3$. In degree 4 they are
$e^4,\,e^4+e^2f,\, e^4+2e^2f+f^2$.

In Step 6 of the algorithm we compute $\Theta(\lambda)$ and $\epsilon(\lambda)$
for the highest weights $\lambda$. First we compute the Weyl group element
$w$ of Section \ref{sec:invol}. In the context of Lemma \ref{lem1} we can
take $l=2$ and $\beta_1\,=\,\alpha$, $\beta_2\,=\,\beta$. Hence $w\,=\, s_{\alpha}s_\beta$.
Then $w(\alpha)
\,=\,-\alpha$, $w(\beta)\,=\,-\beta$. Hence $w(\mu) \,=\, -\mu$ for all weights $\mu$.

Let $\lambda$ be the highest weight of an irreducible $\g^\C$-module. Then
$\lambda(H_i)$ are (non negative) integers. Hence $\lambda^\sigma(H_i)\, =\,
\overline{\lambda(\sigma(H_i))}\,=\, \overline{\lambda(-H_i)}\, =\, -\lambda(H_i)$,
whence $\lambda^\sigma\, =\, -\lambda$ and $\Theta(\lambda)\, =\, w^{-1} \lambda^\sigma\,
=\,\lambda$. 

Let $\lambda$ be a highest weight that we have just computed. 
Write $m_i \,=\, \lambda(H_i)$; then $m_1\,=\,m_2$. As $\epsilon(\lambda) \,=\,
(-1)^{m_1+m_2}$ it follows that the Cartan index is equal to 1.
In other words, all irreducible summands are of type I. So in the algorithm
we have to execute Step 6(a) for each highest weight that we computed.
It is cumbersome to compute the element $\omega_\rho$ by hand. However, in this
case we do not need it. Indeed,
because all highest weights occur with multiplicity one, the corresponding
irreducible submodules are self-conjugate. So if $V^\C$ is such an irreducible
submodule with highest weight vector $v$ then the real and imaginary parts of
$v$ also lie in $V^\C$. At least one of these parts is non-zero and denote it
by $\widehat v$. Then $\g \cdot \widehat{v}$ is the irreducible $\g$-module corresponding
to $\lambda$.

Below we give the decomposition of the $\g$-modules consisting of the
homogeneous polynomials of degree $d$ for $d\,=\,2,\,3,\,4$. 

a. Quadratic Polynomials:
$$\langle \g\cdot(x^2+y^2+z^2+w^2)\rangle \bigoplus \langle \g\cdot
(w^2-y^2)\rangle \, .$$

b. Cubic Polynomials: 
$$
\langle \g \cdot (w^3-3wy^2) \rangle \bigoplus \langle \g\cdot w(x^2+y^2+z^2+w^2) \rangle.
$$

c. Quartic Polynomials:
$$
\langle \g \cdot (w^4-6w^2y^2+y^4) \rangle \bigoplus \langle \g \cdot (w^2-y^2) (w^2+x^2+y^2+z^2) \rangle \bigoplus \langle \g \cdot (w^2+x^2+y^2+z^2)^2 \rangle\, .
$$

\subsection{Implementation and experimental results}\label{sec:impl}

As mentioned in Section \ref{sec:comp}, we have implemented the algorithm
using the functionality of the {\sf CoReLG} (\cite{corelg}) package for
{\sf GAP}4, \cite{gap4}.
A real semisimple Lie algebra
as constructed by that system has Cartan subalgebras that in general are split
over $\Q(\sqq)$ (and not over $\Q$). The algorithm works with an arbitrary
Cartan subalgebra $\cc$ of $\g$. Now
a maximally noncompact Cartan subalgebra $\cc$ tends to have a
basis with elements whose eigenvalues are of small complexity (meaning that
they can be expressed as $a+b\sqq$ with $a$, $b$ rational numbers with
small numerator and denominator), compared to the Cartan subalgebras with
higher compact dimensions. For this reason we use a maximally noncompact
Cartan subalgebra in our implementation.

The most memory intensive step of the algorithm is to find the highest
weight vectors in the $\g^\C$-module $V^\C$. For this reason the current
implementation of the algorithm can be applied to modules of dimensions up
to about a few thousands (depending on the computer one uses, of course). 

We have executed our algorithm on some test examples. In these examples
the module is always of the form $V=W\otimes W$ with $W$ an irreducible
$\g$-module. The run-times \footnote{The computations were done on a
2000 MHz processor with 4GB of memory for {\sf GAP}.}
are displayed in Table \ref{tab:times}. The first
column of this table has the isomorphism type of $\g$. The second column
displays the type of $V$ and the third column has the highest weight of
the corresponding complex representation. The coordinates of these weights
correspond to the nodes of the Dynkin diagram; for this the standard
(Bourbaki-) enumeration of these nodes is used (see \cite[Theorem 2.8.5]{deG2}).
The fourth column displays the dimension of $V$, and the fifth column
has the number of its irreducible summands. The next three columns show the
running times in seconds of some important parts of the algorithm. The
column labeled ``hw'' has the time spent to find the highest weight vectors
of the $\g^\C$-module $V^\C$. 
The column labeled ``basis'' shows the total time spent to find bases of
irreducible $\g^\C$-submodules of $V^\C$ given their highest weight vectors.
(The algorithm used for this is based
on the path model, as indicated in Remark \ref{rem:basis}.)
The column ``$\R$-form'' displays the time taken to find $\R$-bases of
conjugation invariant subspaces of $V^\C$. The last column has the total
time taken.

We see that the algorithm can successfully deal with
representations of dimensions up to around 3000. We also see that the
proportion of the total time taken by the various parts of the algorithm
can vary from case to case. For example, in the computation for
$\mathfrak{so}(6,3)$ most time was taken by computing bases of irreducible
$\g^\C$-modules of $V^\C$, whereas for $E_{7(-25)}$ the bottleneck was the
computation of the $\R$-basis of conjugation invariant subspaces.
All these parts of the algorithm are performed using basic linear algebra.
Their running times depend on the ``complexity'' of the vectors the
program has to deal with (that is, the size of their coefficients with respect
to the basis that is being used). 

\begin{table}[htb]

\begin{tabular}{|r|r|r|r|r|r|r|r|r|}
\hline
$\g$ & type $W$ & highest weight & $\dim W\otimes W$ & \# summands &
hw & basis & $\R$-form & time\\
\hline
$\mathfrak{so}(6,3)$ & II & $(0,0,0,1)$ & 1024 & 20 & 25 & 84 & 15 & 132\\
$\mathfrak{sp}(2,3)$ & I & $(0,1,0,0,0)$ & 1936 & 6 & 148 & 381 & 854 & 1413\\
$E_{6(2)}$ & III & $(1,0,0,0,0,0)$ & 2916 & 9 & 170 & 43 & 938 & 1420\\
$E_{7(-25)}$ & I & $(0,0,0,0,0,0,1)$ & 3136 & 4 & 405 & 638 & 4802 & 5948 \\
\hline
\end{tabular}
\caption{Runtimes in seconds on of the main algorithm on some sample inputs.}\label{tab:times}
\end{table}

We have also implemented the algorithm described in Remark \ref{rem:symbolic}
to compute branching rules. Here we just deal with the irreducible modules
in terms of their highest weights. Therefore we can deal with representations
of significantly higher dimensions. Table \ref{tab:branch} has the running
times on a few sample inputs. The first two columns give, respectively, the
simple real Lie algebra $\a$ and its subalgebra $\g$. (These subalgebras
have been determined in \cite{GM}; the package {\sf CoReLG} has functionality
for reconstructing them.) The next two columns
have the type of $V$ and its highest weight as $\a$-module. The fifth
column displays the number of irreducible modules appearing in the decomposition
of $V$ as $\g$-module. The last column has the running time in seconds
of the algorithm. 

\begin{table}[htb]

\begin{tabular}{|r|r|r|r|r|r|r|}
\hline
$\a$ & $\g$ & type $V$ & highest weight & $\dim V$ & \# summands & time\\
\hline
$E_{6(-14)}$ & $F_{4(-20)}$ & I & $(1,1,1,1,1,1)$ & 68719476736 & 1268 & 14\\
$E_{7(7)}$ & $\mathfrak{su}(4,4)$ & I & $(1,1,0,1,0,0,1)$ & 148780635003 &
5096 & 115\\
$E_{8(8)}$ & $\mathfrak{su}(2)\oplus E_{7(-5)}$ & I & $(0,1,0,0,0,0,1,1)$ &
234550030000 & 1495 & 126\\
$E_{8(-24)}$ & $\mathfrak{so}^*(16)$ & I & $(1,0,0,1,0,0,0,1)$ &
885607531900000 & 20904 & 1053\\
\hline
\end{tabular}
\caption{Runtimes in seconds on of the algorithm for computing branching rules on some sample inputs.}\label{tab:branch}
\end{table}

\section*{Acknowledgments}

One of us (HA) thanks Professors K. H. Neeb, A. W. Knapp and
T. N. Venkataramana for very helpful letters. The third-named author is supported
by a J. C. Bose Fellowship.

\end{document}